\documentclass[11pt,reqno]{amsart}

\usepackage{todonotes}
\usepackage{amsmath}
\usepackage{amsfonts}
\usepackage{mathrsfs}  
\usepackage{amssymb}
\usepackage{bbm}
\usepackage{amsthm}
\usepackage{hyperref}
\usepackage{graphicx}
\usepackage[noadjust]{cite}
\usepackage{caption}
\usepackage{subfig}
\usepackage{dutchcal}
\usepackage{bbm}
\usepackage{tikz}
\usetikzlibrary{decorations.pathreplacing}
\usepackage[T1]{fontenc}
\usepackage[utf8]{inputenc}
\usepackage{mathtools}
\usepackage{pdfpages}

\usepackage[top=2.9cm, bottom=2.9cm, left=2.5cm, right=2.5cm, headsep=0.25in]{geometry}

\usepackage{fancyhdr}

\DeclareMathOperator{\dist}{dist}

\DeclareMathOperator{\Area}{Area}

\theoremstyle{plain}
\newtheorem{theorem}{Theorem}[section]
\newtheorem{defi}[theorem]{Definition}
\newtheorem{corollary}[theorem]{Corollary}
\newtheorem{lemma}[theorem]{Lemma}
\newtheorem{notation}[theorem]{Notation}
\newtheorem{remark}[theorem]{Remark}
\newtheorem{question}[theorem]{Question}

\newtheorem{proposition}[theorem]{Proposition}
\newtheorem{obsv}[theorem]{Observation}
\newtheorem{fact}[theorem]{Fact}

\newtheorem{convention}[theorem]{Convention}

\newcommand{\V}{\Vert}
\newcommand{\RR} {\mathbb R}
\newcommand{\Prob} {\mathbb P}
\newcommand{\CC} {\mathbb C}
\newcommand{\Exp} {\mathbb E}

\newcommand{\NN} {\mathbb N}

\newcommand{\NNN}{\mathcal{N}}

\newcommand{\pa} {\partial}
\newcommand{\Cal} {\mathcal}

\newcommand{\beq} {\begin{equation}}
\newcommand{\eeq} {\end{equation}}

\newcommand{\diam}{\operatorname{diam}}
\newcommand{\inrad}{\operatorname{inrad}}
\newcommand{\capacity}{\operatorname{cap}}

\numberwithin{equation}{section}

\begin{document}
\title{
Heat profile, level sets and hot spots of Laplace eigenfunctions}
\author{Mayukh Mukherjee and Soumyajit Saha}
\address{Indian Institute of Technology Bombay, Powai, Maharashtra 400076, India}
\email{mathmukherjee@gmail.com}
\email{ssaha@math.iitb.ac.in}

\maketitle

\begin{abstract}
  We use probabilistic tools based on Brownian motion and Feynman-Kac formulae to investigate 
  the heat profile for the ground state Dirichlet and second Neumann eigenfunctions. Among other topics, we comment on supremum norm bounds
for ground state Dirichlet eigenfunctions and look at the corresponding Neumann problem, namely the comparison of maximum temperatures on the interior and the boundary, the latter being partially motivated by the hot spots problem.  We also investigate the proximity/distance of level sets of ground state Dirichlet eigenfunctions, some with analogous statements for Neumann eigenfunctions. 
 Domains with {\em bottlenecks} make occasional appearances as an illuminating example as well as  testing ground for our theory. 
\end{abstract}


\section{Introduction and main results} Let $(M, g)$ be a compact Riemannian manifold.  Consider the eigenequation
\begin{equation}\label{eqtn: Eigenfunction equation}
    -\Delta \varphi = \lambda \varphi, 
\end{equation}
where $\Delta$ is the Laplace-Beltrami operator given by (using the Einstein summation convention) 
$$
\Delta f = \frac{1}{\sqrt{|g|}}\pa_i \left( \sqrt{|g|}g^{ij}\pa_j f\right),
$$
where $|g|$ is the determinant of the metric tensor $g_{ij}$. In the Euclidean space, this reduces to the usual $\Delta = \pa_1^2 + \dots + \pa_n^2$. Observe that we are using the analyst's sign convention for the Laplacian, namely that $-\Delta$ is positive semidefinite. If $M$ has a boundary, we will consider either the Dirichlet boundary condition
\begin{equation}\label{eqtn: Dirichlet condition}
    \varphi(x)=0, \hspace{5pt} x\in \pa M,
\end{equation}
or the Neumann boundary condition 
\begin{equation}\label{eqtn:Neumann_condition}
    \pa_\eta\varphi(x)=0, \hspace{5pt} x\in \pa M,
\end{equation}
where $\eta$ denotes the outward pointing unit normal on $\pa M$.
Recall that 
the Dirichlet Laplacian $-\Delta_g$ has a discrete spectrum 
$$ 
0\leq\lambda_1 \leq \dots \leq \lambda_k \leq \dots \nearrow \infty,
$$
repeated with multiplicity with corresponding (real-valued $L^2$ normalized) eigenfunctions $\varphi_k$. If $M$ has a reasonably regular boundary, a similar statement holds for the Neumann Laplacian, whose eigenvalues we denote by $\mu_j$. 

Also, let  $\mathcal{N}_{\varphi_{\lambda}} = \{ x \in M: \varphi_{\lambda} (x) = 0\}$ denote the nodal set of the eigenfunction $\varphi_{\lambda}$. Sometimes for ease of notation, we will also denote the nodal set by $\NNN(\varphi_\lambda)$. Recall that any connected component of $M \setminus \mathcal{N}_{\varphi_{\lambda}}$ is known as a nodal domain of the eigenfunction $\varphi_{\lambda}$ denoted by $\Omega_{\lambda}$. These are domains where the eigenfunction is not sign-changing (this follows from the maximum principle). Recall further that the nodal set is the union of a $(n - 1)$-dimensional smooth hypersurface and a ``singular set'' that is countably $(n - 2)$-rectifiable (\cite{HS89}). 

When two quantities $X$ and $Y$ satisfy $X \leq c_1 Y$ ($X \geq c_2Y$) for constants $c_1, c_2$ dependent on the geometry $(M, g)$, we write $X \lesssim_{(M, g)} Y$ (respectively $X \gtrsim_{(M, g)} Y$). 
	Unless otherwise mentioned, 
	these constants will in particular be independent of eigenvalues $\lambda$. Throughout the text, the quantity $\frac{1}{\sqrt{\lambda}}$ is referred to as the wavelength and any quantity (e.g. distance) is said to be of sub-wavelength (super-wavelength) order if it is $\lesssim_{(M, g)} \frac{1}{\sqrt{\lambda}}$ (respectively $\gtrsim_{(M, g)} \frac{1}{\sqrt{\lambda}}$). 

In this article, we start with finding a lower bound on the supremum of the first non-trivial Dirichlet eigenfunctions on flat/curved ``domains''. In case of domains with Dirichlet boundary, it is clear that the max point is always in the interior of the domain. An interesting question to ask here: how deep does the max point lie? This was addressed in \cite{GM2018} for nodal domains of high energy eigenfunctions. Another interesting line of investigation is to find the position of max points of the first non-trivial Neumann eigenfunctions. This leads to a conjecture (introduced by Rauch in 1974; see \cite{BB99}) famously known as the hot spots conjecture. The max/min points of the first non-trivial (Dirichlet/Neumann) eigenfunction is referred to as hot/cold spots.  
At any rate, crucial to the understanding of hot spots is the study of the level sets and in this paper we particularly focus on qualititative properties of the level and superlevel sets of low energy Dirichlet and Neumann eigenfuntions.

For the purposes of the present note, we are in  particular interested in phenomena of heat diffusion and how a deterministic diffusion process can be expressed as an expectation over the behavior of some random variables in terms of Brownian motion. Although our central theme for the ideas is similar while studying both the Dirichlet and Neumann domains, the treatment or execution of these ideas has to diverge in details owing to the essential differences in the boundary conditions of the domains. One basic difference in the two processes is that while considering domains with Dirichlet boundary, the Brownian particle stops once it hits the boundary, but in the case of Neumann boundary, one considers Reflected Brownian Motion (the Brownian particle gets reflected after hitting the boundary). Our aim with this paper is to discuss  some general tools to tackle both boundary problems utilising the Feynman-Kac formula and martingale properties, and to highlight some common vital ideas and also point out at the appropriate places the crucial differences that leads to contrasting results.  


In the recent work \cite{GM2021}, the authors study the above mentioned  diffusion process and some of its implications on high energy eigenfunctions. Here we continue the discussion along similar lines. Our paper is mainly divided into two sections, separating the study on the level sets of Dirichlet and Neumann eigenfunctions. Below we provide a section-wise overview of this paper and list down the main results.

\subsection{Overview of the paper, and the main results}
In Section \ref{sec:FK_Prelim}, we first recall the basics of Feynman-Kac, hitting probabilities in Euclidean spaces and their behaviour vis-a-vis curvature, local geometry etc. Since some of our results also go through on closed Riemannian manifolds, of particular interest are the sketches of the proofs of comparability of hitting probabilities between Euclidean and curved spacetimes at short time-distance scales, which are abridged from \cite{GM2018}. Though these are pivotal for successfully carrying out  Brownian motion calculations to curved spaces, we could not find them written down explicitly in the literature. Further details can be found in \cite{GM2018}, which also expound some ideas which are implicit in \cite{BePePe95}. In Subsection \ref{subsec:Theta}, we define and recall some elementary estimates of the function $\Theta_n\left(r^2/t\right)$, which was introduced in \cite{GM2021}. We believe this function might also be of interest in future studies of nodal geometry via probabilistic techniques.

We now state our main new results. We also note explicitly that our domains are not simply-connected unless otherwise stated.

In Section \ref{subsec:convex_max_points}, we focus on the ground state Dirichlet eigenfunction. To get the ball rolling, in Proposition \ref{prop:log_concave_p_t} and Corollary \ref{cor:convex_dom_hot_spot_unique} we give our own proofs of a couple of well-known results like log-concavity of survival probability and uniqueness of the point of maximum of the ground state eigenfunction on a convex domain. Next, we prove a supremum norm estimate for a global eigenfunction on a closed Riemannian manifold on each nodal domain, which might be seen as a variant of a supremum norm estimate in \cite{DF1988} for wavelength scale balls. To wit, we prove that the supremum of an eigenfunction satisfies an exponentially small bound in each nodal domain:
\begin{theorem}\label{thm:DF_ext}
    Let $M$ be a closed Riemannian manifold and $\Omega$ be a nodal domain for the eigenfunction $\varphi_\lambda$ for $\lambda \geq C(M, g)$. 
    Let  $x_0 \in \Omega$ be a point of maximum of $|\varphi_\lambda|$ inside $\Omega$. Then,  we have  that 
    \beq\label{ineq:new_sup_bound}
    |\varphi_\lambda(x_0)| \geq ce^{-C\sqrt{\lambda}}\|\varphi_\lambda\|_{L^\infty(M)},
    \eeq
    where $c, C$ are positive constants depending on the geometry $(M, g)$. 
\end{theorem}
The proof of Theorem \ref{thm:DF_ext} is relatively short, but it depends on two deep results: the almost embeddability of wavelength radius balls (as in \cite{GM2018}) and the Remez type inequality in \cite{LM2018, LM2020}. Also, in the special case of Euclidean domains $M$ and corresponding ground state Dirichlet eigenfunctions, (\ref{ineq:new_sup_bound})  can be concluded by using previous work of Lieb \cite{Li83}, and this would hold for all eigenvalues. As a bi-product of Theorem \ref{thm:DF_ext}, we indicate how to recover the best known inner radius bounds for nodal domains in dimension $n \geq 3$ (see \cite{Man_CPDE, GM2018}).

In Subsection \ref{sec:level_Set}, we start looking at more intricate metric properties of level sets.  First, developing on an idea in \cite{St2021, GM2021}, we prove (see Proposition \ref{thm:escape_prob_max_pt}) that a Brownian particle starting at a level set $L_\alpha$ 
(see Notation \ref{not:lev_sub_lev} below) takes sufficient time to exit the domain $\Omega$. 
In view of quasi-conformal ideas which hold in dimension $2$ (see Proposition \ref{prop:planar_escape_prob} below), this gives the following immediate corollary:
\begin{corollary}\label{cor:level_set_bound_dim_2}
    Let $M$ be a closed Riemannian surface, and let $\Omega \subseteq M$ be a domain. Let $\lambda$, $\varphi_\lambda$ be respectively the ground state Dirichlet eigenvalue and eigenfunction of $\Omega$. Then, there exists a constant $c := c(\eta, M, g)$ such that if $\lambda \geq \lambda_0(M, g)$, then
    \beq
    \dist \left(L_\eta, \pa \Omega\right) \geq c\lambda^{-1/2}.
    \eeq
    If $\Omega$ is a planar domain, then $\lambda_0$ can be any positive number, and $c := c(\eta)$ a universal constant.
\end{corollary}
\noindent 
In particular, this extends an old result of Hayman (see \cite{Hay77/78} and further work in \cite{Cr81, Man_Bull, Oss77, Tay79}) about the wavelength inner radius of two dimensional  domains (recall that the dimension $2$ case was not addressed as a corollary of Theorem \ref{thm:DF_ext} above), and also gives an extension to the setting of curved spaces. We note in passing that well-known elliptic methods give pointwise bounds $\| \nabla \varphi_\lambda\|_{L^\infty} \lesssim \sqrt{\lambda}\| \varphi_\lambda\|_{L^\infty}$ on wavelength balls, but this would not give Corollary \ref{cor:level_set_bound_dim_2}.

Next, we start investigating the interspacing of level sets using ideas derived from the Feynman-Kac formula. Among others, we prove the following result here (which is in spirit somewhat antithetical to Corollary \ref{cor:level_set_bound_dim_2} above):
\begin{theorem}\label{thm:Dirichlet_level_set_converse}
Let $M$ be a closed smooth manifold of dimension $n \geq 2$. Given a domain $\Omega$ 
 $\subseteq M$, consider the ground state Dirichlet eigenfunction $\varphi_\lambda$ of $\Omega$. Let $\| \varphi_\lambda\|_{L^\infty(\Omega)} = 1$. Fix two numbers $0 < \mu, \eta \leq 1$. 
Then, we can find a constant $c_1 = c_1\left(\frac{\eta}{\mu}, 
M, g\right)$ such that 
the $c_1\lambda^{-1/2}$-tubular neighbourhoods of $L_\mu$ and $L_\eta$ intersect in a set of positive capacity.  
\end{theorem}
The proof uses a somewhat novel approach involving Doob's optional stopping theorem. More colloquially, the above result indicates that two level sets are at most wavelength distance away upto a constant that depends only on the ratio of the levels. The appearance of the wavelength factor $\lambda^{-1/2}$ makes the result invariant under scaling of the domain. Observe that Theorem \ref{thm:Dirichlet_level_set_converse} can be interpreted as an ``aggregated'' lower bound on the size of $\nabla \varphi_\lambda$, in some sense opposite the pointwise elliptic estimates quoted above. 

The concluding result of Section \ref{subsec:convex_max_points} deals with a qualitative estimate on the ``narrowness'' of superlevel sets. This can be interpreted as another perspective on the proximity or inter-spacing between level sets, complimentary to Theorem \ref{thm:Dirichlet_level_set_converse}.  

\begin{theorem}\label{thm:hot_spot}
    Let $\Omega$ be an $n$-dimensional Euclidean domain, $n\geq 2$ and $u$ be any Laplace eigenfunction on $\Omega$. 
    Take a connected component $S$ of the closure of the $\eta$-superlevel set $\overline{S^c_\eta}$ of $u$, where $\eta > 0$. Then, any $y\in S$ is at a distance at most $\kappa \diam(\Omega)$ from the level set $L_\eta\cap S$, where $\kappa$ is a universal constant given by (\ref{eq:kappa_est}) below.
\end{theorem}
 
 
 A few comments are in order. Strictly speaking, the proof Theorem \ref{thm:hot_spot} does not actually see the global boundary condition, since we use the martingale properties of the backward heat equation only on a superlevel set where the eigenfunction is positive; hence it applies equally well to both Dirichlet and Neumann eigenfunctions (and eigenfunctions without any boundary conditions as well, provided $S$ does not intersect $\pa\Omega$). Observe that one particular corollary of Theorem \ref{thm:hot_spot} is that the inner radius of the $\eta$-superlevel set satisfies 
 \beq \label{ineq:inrad_super}
 \inrad (S) \leq \kappa \diam(\Omega),
 \eeq
which can be contrasted with the inner radius estimates obtained from Theorem \ref{thm:DF_ext} and Corollary \ref{cor:level_set_bound_dim_2}.


Now we move to Section \ref{sec:Neumann}, where we start to focus on Neumann eigenfunctions. We note in passing that using ideas similar to the proof of Theorem \ref{thm:Dirichlet_level_set_converse}, we  prove an analogous result for the Neumann boundary condition:
\begin{theorem}\label{thm:opt_stop_level}
    Given a domain $\Omega$ sitting inside a closed Riemannian manifold $M$, and given two numbers $0 < \nu \leq \eta \in (0, 1]$ and the second normalised Neumann eigenfunction $\varphi$ corresponding to the eigenvalue $\mu$, there exists a universal constant $c\left(\frac{\nu}{\eta}, n\right)$ such that 
    $\dist(L_\eta, L_\nu) \leq c\mu^{-1/2}$. 
\end{theorem}
Since the proof is almost verbatim similar, we skip it. 

As a further application of Doob's optional stopping theorem, we investigate the spectacular lack of decay of Neumann eigenfunctions in domains with narrow horns/tentacles/narrow connectors in Theorem \ref{thm:Neumann_decay}. This will 
establish in a somewhat specific but illustrative situation an important difference in the case of the Neumann boundary condition.  Namely, high level sets (or, points of high temperature) do not need to be surrounded by a large chunk inside the domain (compare with Corollary \ref{cor:level_set_bound_dim_2} above): in fact such points can tunnel into small crevices of the domain. Here is an illustrative statement: 
\begin{theorem}\label{thm:Neumann_decay}
   Consider a tentacle or horn $H$ of the domain $\Omega \subset \RR^n,\; n \geq 2$, and let $\varphi$ be a Neumann eigenfunction which does not vanish inside $H$. Let $\mu$ be the eigenvalue corresponding to $\varphi$. Then for any $x \in H$, we have that 
   \beq
   \varphi (x) \gtrsim \eta 
   e^{\mu N^2}
   \eeq
   up to universal constants, where $\eta$ is a level set $L_\eta$ near the entry of the tentacle $H$, and $N = \dist (x, L_\eta)$.
\end{theorem}
This is in stark contrast to the results in the literature for the Dirichlet case, e.g., Theorem 1.6 of \cite{GM2021}. 

In fact, we now show that points of high temperature not only {\em can} come close to the boundary, but in some sense they actually {\em have} to. In other words, the highest temperature on the boundary is unlikely to be too low compared to the global max. We start with a definition: given a domain $\Omega$, let us define the hot spots constant by the ratio 
$$
\frac{\| u \|_{L^\infty(\pa \Omega)}}{\| u \|_{L^\infty( \Omega)}}.
$$
Let $\lambda_1, 
\mu_2$ denote the first Dirichlet eigenvalue and second Neumann eigenvalue respectively of a domain $\Omega$. Then, we have the following:
\begin{theorem}\label{thm:hot_spot_constant_high_mu}
    The higher is the ratio $\lambda_1/\mu_2$, the closer the hot spots constant is to $1$ (quantitative estimates are given in (\ref{ineq:hot_spot_opti}), (\ref{ineq:hot_spot_c_large}) below). 
\end{theorem}

As a particular illuminating example, we bring in the domain with bottleneck again. It is known that on dumbbell domains with narrow enough bottlenecks (see the description above Lemma \ref{lem:u_equidist} below) the second Neumann eigenvalue is arbitrarily small (see Subsection \ref{subsec:sec_eig_est} below). Whereas, the Dirichlet eigenvalue is controlled because the eigenfunction tends to leave the ``handle'' or the ``connection'' of the dumbbell and localise on the two ends. This gives us the following 
\begin{corollary}\label{cor:dumb_hot_spot_1}
 Let $\Omega$ be a dumbbell domain with narrow enough bottleneck. Then the hot spot constant of $\Omega$ is arbitrarily close to $1$.
\end{corollary}

It is instructive to compare Corollary \ref{cor:dumb_hot_spot_1} with the doubly-connected counterexample of the hot spot conjecture in \cite{Bur05}. We remind the reader that Burdzy's counterexample is also a dumbbell type domain with a narrow bottleneck and the hot spot constant is strictly less than $1$. This shows that the failure of the hot spot conjecture on this domain is rather subtle, and also implies that Corollary \ref{cor:dumb_hot_spot_1} is unlikely to be qualitatively improved, at least in the generality that it is stated here. However, it also turns out that the hot spot in the Burdzy-type counterexample is formed at the middle of a flat plateau, because of which the phenomenon is quite stable under perturbation. Recently, this has been verified through extensive numerical studies by Kleefeld (see for example, Figures 15-21 of \cite{Klee21}).

\section{Feynman-Kac, hitting probabilities and a new function $\Theta_n$}\label{sec:FK_Prelim}

We begin by stating a Feynman-Kac formula for open connected domains in compact manifolds for the heat equation with Dirichlet boundary conditions. Such formulas seem to be widely known in the community, but an explicit reference seems hard to find. An excellent reference for much of the background material is \cite{Si79}. However, here we state the results in the specific contexts/settings that we need them in. The essential content of this subsection can be found in \cite{GM2018}, and we include the outline here for the sake of completeness. 
\begin{theorem}\label{F-K-T}
	Let $M$ be a compact Riemannian manifold. For any open connected $\Omega \subset M$, $f \in L^2(\Omega)$, we have that 
	\beq
	e^{t\Delta} f(x) = \mathbb{E}_x(f(\omega(t))\phi_\Omega(\omega, t)), t > 0, x \in \Omega,
	\eeq
	where $\Delta$ denotes the Dirichlet Laplacian on $\Omega$, $\omega (t)$ denotes an element of the probability space of Brownian motions starting at $x$, $\mathbb{E}_x$ is the expectation with regards to the Wiener measure on that probability space, and 
	$$
	\phi_\Omega(\omega, t) = 
	\begin{cases}
	1, & \text{if } \omega ([0, t]) \subset \Omega\\
	0, & \text{otherwise. }
	\end{cases}
	$$
\end{theorem}
Heuristically speaking, the Dirichlet boundary condition can be encoded by Brownian paths which are killed upon impact at the boundary. However, Theorem \ref{F-K-T} is slightly non-standard as we are on a compact Riemannian manifold (for a general discussion of stochastic processes on curved spaces, we refer the readers to \cite{Hs2002, Str2000}). 


\subsection{Heat content} \label{sec:tools}

We first recall (see \cite{HS89}) that the nodal set can be expressed as a union
\beq
N_\varphi =  H_\varphi \cup \Sigma_\varphi,
\eeq
where $H_\varphi$ is a smooth hypersurface and  
$$
\Sigma_\varphi := \{ x \in N_{\varphi_\lambda} : \nabla \varphi_\lambda(x) = 0 \}
$$
is the singular set which is countably $(n - 2)$-rectifiable. Particularly, in dimension $n = 2$, the singular set $\Sigma_\varphi$ 
consists of isolated points.

Note that we can express $M$ as the disjoint union 
$$ 
M = \bigcup_{j = 1}^{j_0} \Omega_j^+ \cup \bigcup_{k = 1}^{k_0} \Omega_k^- \cup N_{\varphi_\lambda},
$$
where the $\Omega_j^+$ and $\Omega_j^-$ are the positive and negative nodal domains respectively of $\varphi_\lambda$. 

Given a  domain $\Omega \subset M$, consider the solution $p_t(x)$ to the following diffusion process:
\begin{align*}
(\pa_t - \Delta)p_t(x) & = 0, \text{    } x \in \Omega\\
p_t(x) & = 1, \text{    } x \in \pa \Omega\\
p_0(x) & = 0, \text{    } x \in \Omega.
\end{align*}
By the Feynman-Kac formula, this diffusion process can be understood as the probability that a Brownian motion 
particle started in $x$ will hit the boundary within time $t$. As will be clear later on, $p_t(x)$ can be thought of as ``interpolating'' between the characteristic function of the boundary (at time $t = 0$) and the ground state eigenfunction of the domain (at time $ t = \infty$). The quantity 
$$
\int_\Omega p_t(x) dx
$$ 
is called the {\em heat content} of $\Omega$ at time $t$. It can be thought of as a soft measure of the ``size'' of the 
boundary $\pa \Omega$. 

Now, fix an eigenfunction $\varphi_\lambda$ (corresponding to the eigenvalue $\lambda$) and a nodal domain $\Omega$, so that 
$\varphi_\lambda > 0$ on $\Omega$ without loss of generality. Calling $\Delta_\Omega$ the Dirichlet Laplacian on $\Omega$ and 
setting $\Phi (t, x) := e^{t\Delta_\Omega}\varphi_\lambda (x)$, we see that $\Phi$ solves 
\begin{align}\label{eq:Heat_Flow}
(\pa_t - \Delta_\Omega)\Phi (t, x) & = 0, x \in \Omega\nonumber\\
\Phi (t, x) & = 0, \text{  on   } \{\varphi_\lambda = 0\}\\
\Phi (0, x) & = \varphi_\lambda (x), \text{    } x \in \Omega.\nonumber
\end{align}
Using the Feynman-Kac formula we have,
\beq\label{eq:F-K}
e^{-t\lambda}\varphi_\lambda = e^{t\Delta_\Omega}\varphi_\lambda(x) = \mathbb{E}_x(\varphi_\lambda(\omega(t))\phi_{\Omega}(\omega, t)), t > 0, 
\eeq
where $\omega (t)$ denotes an element of the probability space of Brownian motions starting at $x$, $\mathbb{E}_x$ is the 
expectation with regards to the (Wiener) measure on that probability space, and 
$$
\phi_\Omega(\omega, t) = 
\begin{cases}
1, & \text{if } \omega ([0, t]) \subset \Omega\\
0, & \text{otherwise. }
\end{cases}
$$
In particular, $p_t(x) = 1 - \mathbb{E}_x(\phi_\Omega(\omega, t))$.

\subsection{Euclidean comparability of hitting probabilities}\label{Eu-comp}
Implicit in many of our calculations is the following heuristic: if the metric is perturbed slightly, hitting probabilities of compact sets by Brownian particles are also perturbed slightly, provided one is looking at small distances $r$ and at small time scales $t = O(r^2)$. 
Below we make this heuristic precise. 

We first describe the basic set-up. Let $(M, g)$ be a compact Riemannian manifold and cover $M$ by charts $(U_k, \phi_k)$ such that in these charts $g$ is bi-Lipschitz to the Euclidean metric. Consider an open ball $B(p, r) \subset M$, where $r$ is considered small, and in particular, smaller than the injectivity radius of $M$. Let $B(p, r)$ sit inside a chart $(U, \phi)$ and let $\phi (p) = q$ and $\phi (B(p, r)) = B(q, s) \subset \RR^n$. Let $K$ be a compact set inside $B(p, r)$ and let $K^{'} := \phi(K) \subset B(q, s)$. 

For future use, we introduce  the following 
\begin{notation}
    Given a Riemannian manifold $M$ (which includes $\RR^n$ as a special case) and a $K \subseteq M$, $\psi_K^M(T, p)$ denotes the probability that a Brownian motion on $(M, g)$ started at $p$ and killed at a fixed time $T$ hits $K$ within time $T$. In the special case $M = \RR^n$, we will choose to drop the superscript $M$, and write $\psi_K(T, p)$ only.
\end{notation}
Of particular interest (and commonplace in applications) is the special case of $K = \RR^n \setminus B(0, r)$. We denote this particular quantity by $\Theta(r^2/t)$, and we outline some generally useful properties of it in Subsection \ref{subsec:Theta} below. Oftentimes, we will be interested in the situation that the starting point $x$ is inside $\Omega$, and the Brownian particle till time $t$ does not impact on the boundary. This is the same as $1 - \psi_{X\setminus \Omega}(t, x)$, where $X = \RR^n$ or $M$, as the case may be. 
\begin{notation}
    We denote the above mentioned survival probability by $q^\Omega_t(x)$. When there is no scope of confusion, we will drop the superscript $\Omega$, and just write $q_t(x)$. It is clear that $q_t(x) = 1 - p_t(x)$. 
\end{notation}

Now, we fix the time $T = cr^2$, where $c$ is a constant. The following is the comparability result:
\begin{theorem}\label{compare}
	There exists constants $c_1, c_2$, depending only on $c$ and $M$ such that 
	\beq\label{comp1}
	c_1 \psi_{K^{'}}^e(T, q) \leq \psi_K^M(T, p) \leq c_2 \psi_{K^{'}}^e(T, q).\eeq
\end{theorem}

The proof uses the concept of Martin capacity (see ~\cite{BePePe95}, Definition 2.1):
\begin{defi}
	Let $\Lambda$ be a set and $\mathcal{B}$ a $\sigma$-field of subsets of $\Lambda$. Given a measurable function $F : \Lambda \times \Lambda \to [0, \infty]$ and a finite measure $\mu$ on $(\Lambda, \Cal{B})$, the $F$-energy of $\mu$ is 
	\[
	I_F(\mu) = \int_{\Lambda}\int_{\Lambda}F(x, y)d\mu(x)d\mu(y).
	\]
	The capacity of $\Lambda$ in the kernel $F$ is 
	\beq
	\text{Cap}_F(\Lambda) = \left[ \inf_\mu I_F(\mu)\right]^{-1},
	\eeq
	where the infimum is over probability measures $\mu$ on $(\Lambda, \Cal{B})$, and by convention, $\infty^{-1} = 0$. 
\end{defi} 

Now we quote the following general result, which is Theorem 2.2 in ~\cite{BePePe95}.
\begin{theorem}
	Let $\{X_n\}$ be a transient Markov chain on the countable state space $Y$ with initial state $\rho$ and transition probabilities $p(x, y)$. For any subset $\Lambda$ of $Y$, we have
	\beq
	\frac{1}{2}\text{Cap}_M(\Lambda) \leq \mathbb{P}_\rho[\exists n \geq 0 : X_n \in \Lambda] \leq \text{Cap}_M(\Lambda),\eeq
	where $M$ is the Martin kernel $M(x, y) = \frac{G(x, y)}{G(\rho, y)}$, and $G(x, y)$ denotes the Green's function.
\end{theorem} 

For the special case of Brownian motions, this reduces to (see Proposition 1.1 of ~\cite{BePePe95} and Theorem 8.24 of ~\cite{MP2010}):
\begin{theorem}
	Let $\{B(t): 0 \leq t \leq T\}$ be a transient Brownian motion in $\RR^n$ starting from the point $\rho$, and $A \subset D$ be closed, where $D$ is a bounded domain. Then,
	\beq
	\frac{1}{2}\capacity_M(A) \leq \mathbb{P}_\rho\{B(t) \in A \text{ for some } 0 < t \leq T\} \leq \capacity_M(A).\eeq
\end{theorem}

One can check that the above arguments go through with basically no changes on a compact Riemannian manifold $M$, when the Brownian motion is killed at a fixed time $T = cr^2$, and the Martin kernel $M(x, y)$ is defined as $\frac{G(x, y)}{G(\rho, y)}$, with $G(x, y)$ being the ``cut-off'' Green's function defined as follows: if $h_M(t, x, y)$ is the heat kernel of $M$,
\[
G(x, y) := \int_0^T h_M(t, x, y)dt.
\]

Now, to state it formally, in our setting, we have
\begin{theorem}
	\beq
	\frac{1}{2}\capacity_M(K) \leq \psi_K^M(T, p) \leq \capacity_M(K).
	\eeq
\end{theorem}

\subsection{$\Theta_n\left( r^2/t\right)$ and some elementary estimates}\label{subsec:Theta}
The expression $\displaystyle\Theta_n\left( r^2/t\right)$ will appear in many places in our paper. It seems that there is no nice closed formula for $\Theta_n$ in the literature, and we quickly jot down a few facts from \cite{GM2021}. 

Firstly, one estimate that can be quickly calculated from the definition via heat kernel is the following:
\beq
\Theta_n\left(r^2/t\right) = 
\frac{1}{\Gamma\left(\frac{n}{2}\right)}\int_0^t \Gamma\left(\frac{n}{2}, \frac{r^2}{4s}\right) \; ds,
\eeq
where $\Gamma(s, x)$ is the upper incomplete Gamma function defined above.
As discussed in the comparability estimates above, since we will work on curved spaces as well, our main regime of interest is $t \leq r^2$, and $r$ being small. Via asymptotics involving the incomplete Gamma function, given $\varepsilon > 0$, one can choose $r^2/t = c$ large enough to get the following crude upper bound:
\beq\label{ineq:asymp_gamma}
\Theta_n\left(r^2/t\right) \leq \frac{1 + \varepsilon}{\Gamma\left(\frac{n}{2}\right)} c^{\frac{n}{2} - 1} \; t\; e^{-\frac{r^2}{4t}}.
\eeq
Lastly, another crude upper bound for $\Theta_n$ can be obtained via the well-known reflection principle for one-dimensional Brownian particles. Inscribe a hypercube inside the sphere of radius $r$, whence each side of the hypercube has length $\frac{2r}{\sqrt{n}}$. So by the reflection principle, the probability of a Brownian particle escaping the cube is given by 
\beq\label{eq:ref_prin}
\left(2\Phi\left(-\frac{r}{\sqrt{nt}}\right)\right)^n, 
\eeq
where 
\beq\label{normal_cumulative}
\Phi(s) = \int_{-\infty}^s \frac{1}{\sqrt{2\pi }}e^{-\frac{s^2}{2}}\; ds
\eeq
is the cumulative distribution function of the one dimensional normal distribution. On calculation, it can be checked that in the regime $r^2/t \geq n$, the expression in (\ref{eq:ref_prin}) translates to 
\beq\label{ineq:theta_large}
\Theta_n\left(r^2/t\right) \leq \frac{2^{3n/2}}{\pi^{n/2}} e^{-\frac{r^2}{2t}}.
\eeq

\section{
Level sets, Brownian exits and hot spots of the ground state eigenfunction}\label{subsec:convex_max_points}
As an immediate application, we first prove that on a convex domain $\Omega \subset \RR^n$, the level sets of the ground state Dirichlet eigenfunction are convex. This is folklore at this point, but we hope our approach might reveal some insights into the heat theoretic methods. 
The main ideas of the following proposition are implicit in \cite{BL76}. 
\begin{proposition}\label{prop:log_concave_p_t}
Let $\Omega \subset \RR^n$ be a convex domain 
and consider the Brownian motion inside $\Omega$ which is stopped upon impact at the boundary. 
Then, $q_t(x)$ is log-concave. In other words, for all $x, y \in \Omega$ 
one has
\beq
q_t(x)q_t(y) \leq q_t\left(\frac{x + y}{2}\right)^2.
\eeq
\end{proposition}
\begin{proof}
For $x \in \Omega$, by definition, 
\beq
q_t(x) = \int_\Omega p(t, x, z)\; dz,
\eeq
where $p(., ., .) \in \RR^{+} \times \RR^n \times \RR^n$ is the Euclidean heat kernel. The idea is to use the Trotter product formula to discretise the problem. Let $n \in \NN$, and let 
\beq
p_{N,t}(x, z) = \frac{1}{\left(4\pi \frac{t}{N}\right)^{n/2}}e^{- \frac{|x - z|^2}{4 \frac{t}{N}}}, 
\eeq
and let $p_{N, t}(x)$ denote the probability that for $j = 1,\cdots, N - 1$, the Brownian particle is inside $\Omega$ at times $t_j = jt/N$.
In the limit, one has that (for regular enough boundary  $\pa\Omega$)
\beq
\lim_{N \to \infty} p_{N, t}(x) = q_t(x).
\eeq
We also have that (via heat semigroup property)
\begin{align*}
p_{N, t}(x) & = \int_{\Omega^N} p_N(t, x, z_1)\; p_N(t, z_1, z_2)\cdots p_N(t, z_{N - 1}, z) \; dz_1\cdots dz_{N - 1}\; dz.
\end{align*}
By way of notation, let 
\beq
\Psi_{N, t}(x, z) = p_N(t, x, z_1)\; p_N(t, z_1, z_2)\cdots p_N(t, z_{N - 1}, z).
\eeq
Since the functions $|z_i - z_j|^2$ are convex, we see that 
\beq
\Psi_{N, t}(x, z)\;\Psi_{N, t}(y, w) \leq \Psi_{N, t}\left(\frac{x + y}{2}, \frac{z + w}{2}\right)^2.
\eeq
The conclusion now follows via a standard limiting procedure. 
\end{proof}

As is well-known, using the real analyticity of the eigenfunction, 
Proposition \ref{prop:log_concave_p_t} gives us a quick proof of the fact that 
\begin{corollary}\label{cor:convex_dom_hot_spot_unique}
 On a convex Euclidean domain $\Omega$, the ground state eigenfunction of the Dirichlet Laplacian has a unique point of maximum.
\end{corollary}
\subsection{Max point of ground state eigenfunction and supremum norm bounds}\label{sec:supremum_bound} Let $M$ be a closed Riemannian manifold, and let $\Omega_\lambda$ be a nodal domain for a Laplace eigenfunction $\varphi_\lambda$ (assumed positive on $\Omega_\lambda$ without loss of generality), and let $x_0 \in \Omega_\lambda$ satisfy 
$$
\varphi_\lambda(x_0) = \| \varphi_\lambda\|_{L^\infty(\Omega_\lambda)}.
$$

We now prove that on any nodal domain, the maximum value of an eigenfunction cannot become arbitrarily small, which is Theorem \ref{thm:DF_ext} above. To put the result in context, we recall the following insight from \cite{DF1988}:
\begin{proposition}
On any ball $B(x, r) \subseteq M$ where $r \gtrsim \lambda^{-1/2}$ the eigenfunction $\varphi_\lambda$ satisfies
\beq\label{ineq:wav_ball_DF}
\| \varphi_\lambda \|_{L^\infty(B(x, r))} \geq c_1e^{-c_2\sqrt{\lambda}}\| \varphi_\lambda\|_{L^\infty(M)}, 
\eeq
where $c_j := c_j(M, g)$. 
\end{proposition}

Now we give a proof of Theorem \ref{thm:DF_ext}. 

\begin{proof}
The proof 
depends on two results, which we state below.

Let $B \subset M$ be a ball, and $E \subseteq B$ be any set with positive volume. Then the Remez type inequality introduced in \cite{LM2018, LM2020} basically says that $\varphi_\lambda$ cannot be arbitrarily small on $E$, in the following sense:
\begin{theorem}\label{thm:Remez}
    We have that 
    \beq\label{ineq:Remez}
    \sup_{E} |\varphi_\lambda| \geq c\left( \frac{c|E|}{|B|}\right)^{C\sqrt{\lambda}}\sup_{B}|\varphi_\lambda|,
    \eeq
    where $c, C > 0$ depend on the geometry $(M, g)$. 
\end{theorem}
In ~\cite{GM2018}, 
 the authors proved that one can almost fully inscribe a wavelength radius ball inside a nodal domain, in the following sense:
\begin{theorem}\label{alfulins}
	Let $\dim M \geq 3, \epsilon_0 > 0 $ be fixed and $ x_0 \in \Omega_\lambda $ be such that $ |\varphi_\lambda(x_0)| = \text{max}_{\Omega_\lambda}|\varphi_\lambda| $. There exists $ r_0 = r_0 (\epsilon_0) $, such that
	\begin{equation}\label{Vol}
	\frac{|B_{r_0} \cap \Omega_\lambda |}{|B_{r_0}|} \geq 1 - \epsilon_0,
	\end{equation}
	where $ B_{r_0} $ denotes $ B\left(x_0, \frac{r_0}{\sqrt{\lambda}}\right) $.
\end{theorem}
    Using (\ref{ineq:wav_ball_DF}), we have that $$
    \sup_{B(x, \frac{1}{\sqrt{\lambda}})} |\varphi_\lambda| \gtrsim e^{-C\sqrt{\lambda}}\V \varphi_\lambda\V_{L^\infty(M)}.
    $$
Calling $E := B_{r_0} \cap \Omega_\lambda$ from Theorem \ref{alfulins} above and using Theorem \ref{thm:Remez}, it is clear that 

\begin{align*}
    \sup_E |\varphi_\lambda| & \geq \left(c(1 - \epsilon_0)\right)^{C\sqrt{\lambda}}\sup_{B_{r_0}}|\varphi_\lambda|\\
    & \geq c'\;e^{-C\sqrt{\lambda}}\|\varphi_\lambda\|_{L^\infty(M)}.
\end{align*}
\end{proof} 

To put Theorem \ref{thm:DF_ext} in perspective: observe that the proof would be trivial if one were to have embeddability of wavelength balls inside nodal domains (in other words, wavelength inner radius). However, as noted in \cite{GM2019}, the heuristic also goes in the other direction. To put it simply: a higher value of the eigenfunction would give a bigger inner radius. A quantitative estimate of this is recorded in Theorem 1.8 of \cite{GM2019}, which gives us the following 
\begin{corollary}
    Let $M$ be a closed Riemannian manifold of dimension $\geq 3$. On any nodal domain $\Omega_\lambda \subseteq M$, we have that 
    $$
    \inrad\left(\Omega_\lambda\right) \gtrsim \lambda^{-\frac{n - 1}{4} - \frac{1}{2n}}.
    $$
\end{corollary}
This recovers well-known results from \cite{Man_CPDE, GM2018}. 
\begin{remark}
Theorem \ref{alfulins} implies that any nodal domain $\Omega_\lambda$ satisfies $|\Omega_\lambda| \gtrsim \lambda^{-n/2}$, which is a fact trivial on Euclidean domains (because of the Faber-Krahn inequality), but gives a refinement of the Faber-Krahn inequality on curved spaces.
Also, recall that the proof of Theorem \ref{alfulins} is  essentially heat theoretic, which makes the proof of Theorem \ref{thm:DF_ext} an application of heat theoretic techniques. Also, as observed in \cite{Li83}, the proof there does not work readily on manifolds, which is why we resort to Theorem \ref{alfulins}.
\end{remark}

\subsection{Level set separation in dimension two}\label{sec:level_Set}

To begin with, we prove that for a $2$-dimensional domain inside a Riemannian surface $M$, a Brownian particle starting close to the boundary exits the domain quickly. To allow for Euclidean comparability, on a curved setting one has to phrase such a result in terms of wavelength distance scales and wavelength$^2$ time scales, where the frequency is high enough. The case of the planar domain will be automatically subsumed by the proof. 
\begin{proposition}\label{prop:planar_escape_prob} Let $M$ be a smooth Riemannian surface, and let $\Omega \subset M$ be a bounded domain, such that the ground state eigenvalue $\lambda$ of $\Omega$ satisfies $\lambda \geq \lambda_0(M, g)$. Let $r_0, t_0$ be small enough positive constants, and let $x \in \Omega$ be such that $\dist (x, \pa \Omega) \leq r_0\lambda^{-1}$. Then $p_{t_0\lambda^{-1}}(x) \geq C$, where $C$ is a constant whose value depends only on $r_0^2/t_0$.
\end{proposition}
The following proof appears implicitly in \cite{GM2021}, but since it seems quite useful we still include the main ideas here for the sake of completeness.
\begin{proof}
Suppose we have $\text{dist} (x, \pa \Omega) < kr$, where $k$ is a constant. It is known that there exists such an $r$ such that for any eigenvalue $\lambda$ and disk $B \subset M$ of radius $\leq r\lambda^{-1/2}$, there exists a $K$-quasiconformal map $h : B(x, r) \to D \subset \CC$, where $D$ is the unit disk in the plane such that $x$ is mapped to the origin (for more details, see Theorem 3.2 of \cite{Man_Bull}, and also \cite{Nad91}, \cite{NPS05}). 
 Now, define $E : = D \setminus 
h(\Omega \cap B(x, kr))$, and assume that $\psi_E(t, 0) \leq \epsilon \leq \delta$, where $\delta$ is a small enough constant, assumed without loss of generality to be $< 1/2$. We will now investigate the implication of the above assumption on $
w(0, E) $, the harmonic measure of the set $E$ with a pole at the origin. Adjusting the ratio $k^2r^2/t$ suitably depending on $\delta$, we can arrange that the probability of a Brownian particle starting at at the origin to hit the boundary $\pa D$, and hence $h (\pa B(x, kr))$ within time $t$ is at least $1 - \delta$. Setting $\psi_E (\infty, 0)$ to be the probability of the Brownian particle starting at $x$ to hit $E$ after time $t$, we have that 
$$
w(0, E) = \psi_E(t, 0) + \psi_E(\infty, 0) \leq \epsilon + \delta \leq 2\delta.$$
On the other hand, by the Beurling-Nevanlinna theorem  (see \cite{Ahl73}, Section $3$-$3$), 
$$
w(0, E) \geq 1 - c\sqrt{\text{dist} (0, E)}, $$
which shows that 
$$
\text{dist} (0, E) \gtrsim (1 - 2\delta)^2.$$
This proves our contention.
\end{proof}

\vspace{5mm}
 \begin{figure}[ht]
\centering
\includegraphics[scale=0.14]{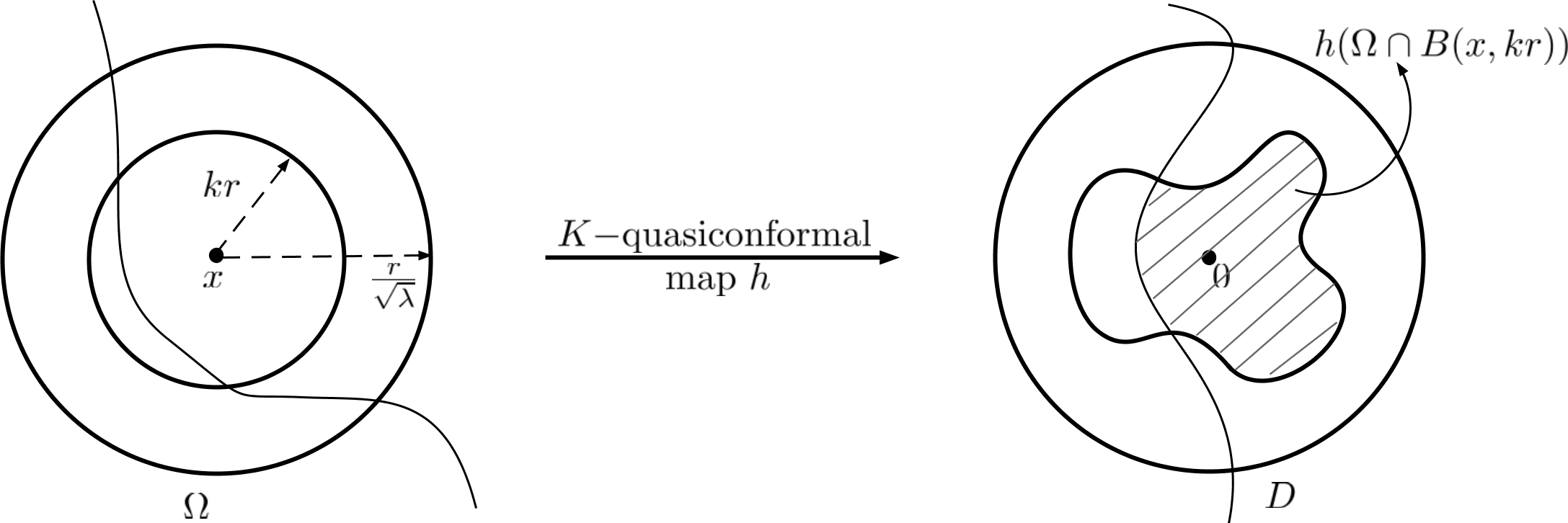}\label{diag3}
\caption{Quasi-conformal mapping on wavelength balls}
\end{figure}
     
    \vspace{5mm}

Observe that the main idea in the proof of Proposition \ref{prop:planar_escape_prob} is that the probability of a Brownian particle which starts at $x \in \RR^2$ hitting an obstacle $U$ within time $t$ has a lower bound in terms of $\dist(x, U)^2/t$. In plain language: a particle close to an obstacle hits the obstacle quickly. 
We note that such a statement is not true in dimensions $n \geq 3$. As an example, one can imagine $x$ being close to a ``sharp spike'' of very low capacity. For a Brownian particle starting at a point inside the nodal domain which is wavelength-near from the ``tip'' of one such spike, the probability of striking the nodal set is still negligible. 
The same phenomenon is at the heart of why in dimension $n = 2$, domains are proven to have wavelength 
inner radius (see \cite{Hay77/78}, and also our ensuing generalisation in Corollary \ref{cor:level_set_bound_dim_2}), but in higher dimensions, one cannot make such a conclusion without allowing for a volume error. 
For more details, refer to \cite{Li83},  \cite{Mu_2021_Wa} and \cite{GM2018} (in the special setting of real analytic manifolds, improvements 
can of course be made; see \cite{Ge2019}, which combines arguments from \cite{JM2009} and \cite{GM2018} and uses rather strongly the 
analyticity of the domain).

Now we prove a quantitative estimate saying that a Brownian particle starting at any max point of the ground state Dirichlet eigenfunction of a domain $\Omega$ (Euclidean, or contained inside a compact Riemannian manifold) takes some time to exit the domain. This has already been used implicitly in proofs in \cite{St_CPDE, GM2018} (and is also implicit in many earlier works of Ba\~{n}uelos, Burdzy et al, for example, see \cite{Bu2016}). Since this seems particularly useful, we record it formally. 

\begin{theorem}\label{thm:escape_prob_max_pt}
   Given $\alpha \in (0, 1]$, let $x \in \Omega$ satisfy $\varphi_\lambda(x) = \alpha \| \varphi_\lambda\|_{L^\infty(\Omega)}$. Then, 
   \beq\label{ineq:thm:escape_prob_max_pt}
   p_t(x) \leq 1 - \alpha e^{-\lambda t}.
   \eeq
\end{theorem}
\begin{proof}
Without loss of generality, we assume that $\varphi_\lambda$ is positive on $\Omega_\lambda$. 
Now, using the Feynman-Kac formula given in (\ref{eq:F-K}) we have, 
\begin{align*}
    \Phi(t, x)= e^{-\lambda t}\varphi_\lambda(x)&= \mathbb{E}_x(\varphi_\lambda(\omega(t))\phi_{\Omega_\lambda}(\omega, t)) \leq  \|\varphi_\lambda\|_{L^{\infty}(\Omega_\lambda)}\mathbb{E}_x(\phi_{\Omega_\lambda}(\omega, t)).
\end{align*}
Using our assumption, we have that 
\begin{equation*}
    \alpha e^{-\lambda t}\|\varphi_\lambda\|_{L^{\infty}(\Omega_\lambda)} =  e^{-\lambda t}\varphi_\lambda(x) \leq \|\varphi_\lambda\|_{L^{\infty}(\Omega_\lambda)} (1- p_t(x)),
\end{equation*}
which gives
\begin{equation}
    p_t(x)  \leq 1- \alpha e^{-\lambda t},
\end{equation}
which proves our claim.
\end{proof}

Now we give a proof for Corollary \ref{cor:level_set_bound_dim_2}. 
\begin{proof}
    By the comparability of hitting probabilities with the Euclidean setting discussed above, it is enough to discuss the flat case. Pick $x \in L_\eta$. By Theorem \ref{thm:escape_prob_max_pt}, it is already known that 
    $$
    p_{t_0\lambda^{-1}}(x) \leq 1 - \alpha e^{-t_0}.
    $$
    Taking $r = r_0\lambda^{-1/2}$, and adjusting the ratio $r_0^2/t_0$, we see from Theorem \ref{prop:planar_escape_prob} that $\dist(x, \pa\Omega) \geq r_0\lambda^{-1/2}.$
\end{proof}

\begin{notation}\label{not:lev_sub_lev}
    Let $\varphi_\lambda$ be the ground state eigenfunction for some domain $\Omega_\lambda$. Then we define the $\delta$-level set $L_\delta := \{x \in \Omega_\lambda : |\varphi_\lambda(x)| = \delta \|\varphi_\lambda\|_{L^\infty(\Omega_\lambda)}\}$. Further, define $S_\delta := \left\{ x \in \Omega : |\varphi_\lambda (x)| \leq \delta \|\varphi_\lambda\|_{L^\infty(\Omega)} \right\}$ as the $\delta$-sublevel set of $\varphi_\lambda$ in $\Omega$. $S_\delta^c$ denotes the complement $\Omega \setminus S_\delta$, the $\delta$-superlevel set of $\varphi_\lambda$ in $\Omega$. 
\end{notation}

Before continuing further, we pose the following 
\begin{question}\label{ques:escape_max_to_level}
    Can one generalise Theorem \ref{thm:escape_prob_max_pt} to give an upper bound for probability that a Brownian particle starting at the max point intersects the level set $L_\delta$ 
    within time $t$? Besides being interesting in its own right, such an estimate, coupled with ideas in the proof of Theorem \ref{alfulins} should give a lower bound on the volumes of superlevel sets, giving an alternative proof of Theorem 1.2.4 of \cite{Pol2017}. 
\end{question}

We now delve more into the interaction of different level sets of ground state eigenfunctions. In a certain sense, the following two results can be seen as variants of Theorem \ref{thm:escape_prob_max_pt} above. To wit, Theorem \ref{thm:escape_prob_max_pt} says that the zero level set and the max level set  are ``distant''. Here we extend the result to arbitrary level sets at levels $\mu$ and $\nu$. The Dirichlet version of the following theorem has already been proved in \cite{GM2021}, we start by proving the Neumann version here. 
\begin{convention}
    In what follows, for the Neumann boundary condition and the corresponding reflected Brownian motion (RBM henceforth), we will consider domains $\Omega$ which will have ``sufficiently regular boundary''. We are not sure what is the most general boundary on which one can construct RBM such that variants of the  Feynman-Kac still hold, but it is known that Lipschitz regularity of the boundary suffices, see \cite{BS91}.
\end{convention}

\begin{proposition}\label{thm:Dirichlet_level_set}
Consider a domain $\Omega$ (with sufficiently regular boundary) sitting inside a closed smooth manifold $M$ of dimension $n$, and let $\varphi$ be a Neumann eigenfunction of $\Omega$ corresponding to the eigenvalue $\mu$. 
Let $\tilde{\Omega}$ be a nodal domain of $\varphi$. Normalise $\| \varphi\|_{L^\infty(\tilde{\Omega})} = 1$ and fix two numbers $0 < \nu < \eta \leq 1$. Now,  choose a point $x \in L_\nu$. 
Then, for any positive number $\tau$, we have that 
\beq\label{ineq:upp_bdd_level_hit}
\psi_{S^c_\eta}(\tau, x) \leq \frac{\nu}{\eta}\left( \frac{1 - e^{-\mu\tau}}{\mu}\right),
\eeq
where $S^c_\eta$ denotes the $\eta$-superlevel set of $\varphi$ inside $\tilde{\Omega}$. 
\end{proposition}
In plain language, the upper bound (\ref{ineq:upp_bdd_level_hit}) says that a particle starting at the $\nu$-level takes some time to hit the $\eta$-level. A precise interpretation of this result does not seem entirely straightforward. Heuristically, to us it seems to imply a sort of a dichotomy: if the level sets $L_\eta$ and $L_\nu$ are close to each other, then $L_\nu$ must be close to the boundary $\pa\Omega$, so that a lot of Brownian particles strike $\pa \Omega \cap \pa\tilde{\Omega}$ first, get reflected from there and then strike $L_\eta$, or get killed upon impact on $\pa \tilde{\Omega} \setminus \pa\Omega$ before reaching $L_\eta$. This makes the Brownian particle travel a longer distance or get killed before impact, thereby lowering the hitting probability.

 \begin{figure}[ht]
\centering
\includegraphics[scale=0.16]{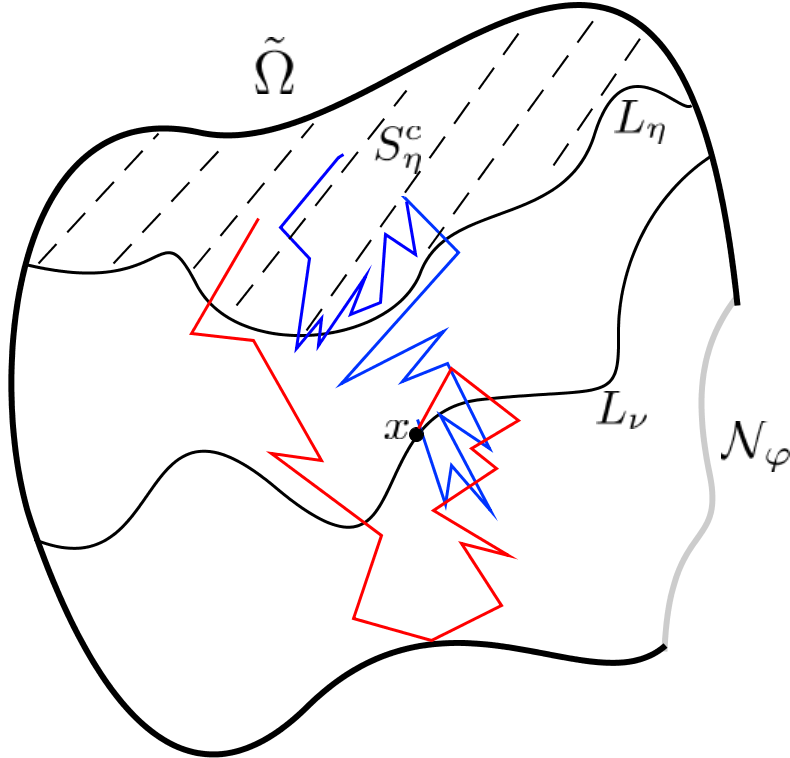}
\caption{Superlevel sets, direct (blue) and reflected (red) Brownian particles}
\end{figure}

\begin{proof}
    The proof is simple. 
    Via the Feynman-Kac formula for the mixed boundary condition, we have that
    \begin{align*}
        e^{-\mu t}\varphi(x) & = \int_{\omega(t) \in S_\eta^c} \varphi  (\omega (t)) 
        \; d\omega  + \int_{\omega(t) \in S_\eta} \varphi  (\omega (t)) 
        \; d\omega\\
        & \geq \eta \int_{\omega(t) \in S_\eta^c}  
        \; d\omega,
    \end{align*}
    where as usual, $d\omega$ means the Wiener measure element on the space of Brownian paths inside $\tilde{\Omega}$ which are reflected upon impact at the ``outside'' boundary and absorbed upon impact at the nodal set. Integrating the above inequality, we get that 
\begin{align*}
\nu \left( \frac{1 - e^{-\mu\tau}}{\mu}\right) & \geq \eta \int_{t = 0}^\tau \int_{\omega(t) \in S_\eta^c}  
\; d\omega \\
& \geq \eta\; \psi_{S^c_\eta} (\tau, x).
\end{align*}
\end{proof}

We now turn 
to the converse result of Proposition \ref{thm:Dirichlet_level_set}. 
Note that we have stated Theorem \ref{thm:Dirichlet_level_set_converse} for the Dirichlet boundary condition. However, the proof for the corresponding statement for the Neumann case is similar. 
To achieve this, we now bring in a tool called optional stopping time (see Chapter 3 of \cite{Ba2011}). This takes advantage of the martingale properties of Brownian processes, and we believe that such ideas might find further applications in related problems in future. 

\begin{proof}[Proof of Theorem \ref{thm:Dirichlet_level_set_converse}] 
Let $x \in \Omega$ lie on a 
$\mu$-level set and consider another level set $L_\eta$. Without loss of generality, let us assume that $\mu > \eta$. 
Let $u(t, x) = e^{\lambda t}\varphi_\lambda(x)$ be the solution of the backward heat equation with 
Dirichlet boundary conditions on 
$\Omega$, as described above. 
Recall that the solution of the backward heat equation 
gives rise to a non-negative martingale $u(t, \omega(t))$, and the time $T$ when a Brownian particle emanating from $x$ hits $L_\eta$ is a stopping time, and so is $t^* := \min\{t, T\}$, where $t > 0$ is some real number. Then by Doob's optional stopping theorem, we have that 
\begin{align*}
\varphi_\lambda(x) & = \Exp_x\left( e^{\lambda.0}\varphi_\lambda(x)\right) = \Exp\left( e^{\lambda t^*}\varphi_\lambda(\omega(t^*))\right)\\
& = \Exp\left( e^{\lambda T}\varphi_\lambda(\omega(T))\chi_{T \leq t}\right) + \Exp\left( e^{\lambda t}\varphi_\lambda(\omega(t))\chi_{T \geq t}\right)\\
& \geq \Exp\left( e^{\lambda T}\varphi_\lambda(\omega(T))\chi_{T \leq t}\right).
\end{align*} 
Letting $t \nearrow \infty$, it is clear that 
\begin{align*}
    \varphi_\lambda(x) & \geq \eta \|\varphi_\lambda\|_{L^\infty(\Omega)} \Exp\left( e^{\lambda T}\right),
\end{align*}
giving finally that
\beq
\frac{\mu}{\eta} \geq \Exp\left( e^{\lambda T}\right).
\eeq
Now recall Markov's inequality which states that for a nonnegative random variable $X$ and a positive number $a$, 
$$
\Prob (X \geq a) \leq \frac{\Exp(X)}{a}.
$$
This gives us that 
\begin{align*}
    \Exp\left( e^{\lambda T}\right) & \geq e^a \Prob \left( e^{\lambda T} \geq e^a\right)\\
    & = e^a \Prob \left( T \geq a/\lambda\right).
\end{align*}
Choosing $a = t_0$ (a constant) such that $t_0> -\ln\left(\frac{\eta}{\mu}\right)$, we see finally that 
$$
\frac{\mu}{\eta} \geq e^{t_0}\Prob\left( T \geq \frac{t_0}{\lambda}\right),
$$

which implies that 
\begin{align*}
    1-e^{-t_0}\frac{\mu}{\eta}\leq \Prob\left(T<\frac{t_0}{\lambda} \right) \leq \Theta_n\left(\frac{\dist(L_\mu, L_\eta)^2}{t_0/\lambda} \right).
\end{align*}

On calculation we have that 
\begin{equation}\label{ineq: level_set_dist_est}
    \dist (L_\mu, L_\eta) \leq \left( \frac{t_0}{\lambda} \Theta_n^{-1}\left(1 - e^{-t_0}\frac{\mu}{\eta}\right)\right)^{1/2} \lesssim_{\mu/\eta} \lambda^{-1/2},
\end{equation}
which gives our claim.
\end{proof}

We note here that if $\eta/\mu$ 
is close to $1$, then choosing $t_0$ appropriately and using (\ref{ineq:theta_large}), we have a more precise description of the constant mentioned in (\ref{ineq: level_set_dist_est}), viz
\begin{equation*}
    \dist(L_\mu, L_\eta)\leq \left(\frac{t_0}{\lambda}\left\{-2\ln{\left[\frac{\pi^{n/2}}{2^{3n/2}}\left(1-e^{-t_0}\frac{\mu}{\eta}\right)\right]} \right\} \right)^{1/2}.
\end{equation*}

\subsection{Proof of Theorem \ref{thm:hot_spot}}
Let $u$ denote the first non-trivial Neumann eigenfunction of $\Omega \subset \RR^n$. 
 
\begin{proof}[Proof of Theorem \ref{thm:hot_spot}]
    Observe that $v(t, x) := e^{t\mu_2}u(x)$ is a solution to the backward heat equation, hence $e^{t\mu_2}u(\omega(t))$ is a martingale. Here $\omega(t)$ is a Brownian particle in 
    $S$ which is reflected on $\pa\Omega \setminus L_\eta$, and stopped on impact at $L_\eta$, which we call $L$ for the ease of notation. Now, we start a Brownian motion at a point $y \in S$. 
    Let $T_L$ denote the first impact time of this particle on $L$, 
    which is a stopping time. 
    This implies that so is $\tau = \min\{ t, T_L\}$. By the optional stopping theorem, we have that 
    \begin{align*}
u(y) & = \Exp_y\left( e^{0.\mu_2}u(y)\right) = \Exp\left( e^{\mu_2 \tau}u(\omega(\tau))\right)\\
& = \Exp\left( e^{\mu_2 T_L}u(\omega(T_L))\chi_{T_L \leq t}\right) + \Exp\left( e^{\mu_2 t}u(\omega(t))\chi_{T_L \geq t}\right)\\
& \geq \Exp\left( e^{\mu_2 T_L}u(\omega(T_L))\chi_{T_L \leq t}\right).
\end{align*} 
Letting $t \nearrow \infty$, it is clear that 
\begin{align*}
    u(y) & \geq \eta\|u\|_{L^\infty(\overline{\Omega})} \Exp\left( e^{\mu_2 T_L}\right),
\end{align*}
which gives us
\beq
\frac{1}{ \eta} \geq \Exp\left( e^{\mu_2 T_L}\right).
\eeq

Now recall Markov's inequality which states that for a nonnegative random variable $X$ and a positive number $a$, 
$$
\Prob (X \geq a) \leq \frac{\Exp(X)}{a}.
$$
This gives us, 
\begin{align*}
    \Exp\left( e^{\mu_2 T_L}\right) & \geq e^a \Prob \left( e^{\mu_2 T_L} \geq e^a\right)\\
    & = e^a \Prob \left( T_L \geq a/\mu_2\right) \\
    & \geq e^a \Prob \left(T_L \geq \frac{a \diam(\Omega)^2}{\tilde{c}}\right),
\end{align*}
where $\tilde{c}$ is a constant satisfying $\mu_2 \geq \tilde{c}/\diam(\Omega)^2$ (we know that $\tilde{c} \leq c_n$ always works). Calling $d = \diam(\Omega)$,  let us denote $\displaystyle\kappa d=  \dist(y, L)$ for some $\kappa$. 
It is clear that 
\begin{align*}
    \Prob\left( T_L < 
    \frac{ad^2}{\tilde{c}} \right) \leq \Theta_n\left( \frac{\dist(y, L)}{ad^2/\tilde{c}} \right) = \Theta_n\left( \frac{\kappa^2 \tilde{c}}{a} \right),
\end{align*}
which implies on calculation that 
\begin{align*}
    1/\eta \geq \Exp \left(e^{\mu_2 T_L}\right) & \geq e^a\left( 1 - \Theta_n\left(\frac{\kappa^2 \tilde{c}}{a}\right) \right). 
\end{align*}

This implies,
\begin{equation*}
     \left(1-\frac{e^{-a}}{\eta}\right) \leq \Theta_n\left(\frac{\kappa^2 \tilde{c}}{a}\right).
\end{equation*}

Choosing $a>-\ln{\eta}$, we have that 
\begin{equation}\label{eq:kappa_est}
  \kappa \leq \inf_{a>-\ln{\eta}} \tilde{c}^{-1/2}\left( a\Theta_n^{-1}\left(1-\frac{e^{-a}}{\eta}\right)\right)^{1/2} = \inf_{b>0} \tilde{c}^{-1/2}\left( -b\ln{(\eta(1-\Theta_n(b)))}\right)^{1/2} 
\end{equation}
\end{proof}

\begin{remark}
    Note that, as $b\to \infty$ or $b\to 0$, the above expression on the right of (\ref{eq:kappa_est}) goes to $\infty$. So the infimum is attained for some $b$ away from the two ``endpoints''. Also, it is clear that in (\ref{eq:kappa_est}) above, $a$ is small iff $b$ is large.  Hence, one way to get a (possibly quite sub optimal) bound of $\kappa$ is the following: in the regime of large $b$ ($b>n$), we bring in the estimate (\ref{ineq:theta_large}) to see that
    
    \begin{align*}
      \tilde{c}^{-1/2}\left( -b\ln{(\eta(1-\Theta_n(b)))}\right)^{1/2}   &\leq c_n^{-1/2}b^{1/2}\left(-\ln\left\{\eta\left(1-\frac{2^{3n/2}}{\pi^{n/2}}e^{-b/2}\right)\right\}\right)^{1/2}.
    \end{align*}

\end{remark}

\section{Neumann eigenvalue bounds, level sets and hot spots} \label{sec:Neumann}
We begin by listing some well-known (or at any rate easily provable) facts and observations:
\begin{fact}
    Any connected component of the nodal set for the first non-trivial Neumann eigenfunction $u$ has to cut the boundary at two distinct points. It cannot be an embedded circle or cut the boundary at precisely one point. 
\end{fact}   
The proof is simple and follows from well-known inequalities between Dirichlet and Neumann eigenvalues (for instance, see \cite{Fr91}). 

\begin{obsv} (Trivial) 
    The maximum on the boundary cannot be strictly greater than the supremum on the interior, just via intermediate value property. However, it might happen that the supremum on the interior is not attained. 
\end{obsv}
    
\begin{obsv} 
    However, if the supremum/infimum is  attained at some interior point, there are two options: the maximum level set is a family of  curves, or a collection of isolated points. Because of the maximum principle, the level set of the maximum cannot be an embedded circle. Now, suppose that we are in dimension $n = 2$, and  the hot spots curve separates the domain into two components. Then on each component the eigenfunction satisfies the Neumann boundary condition, which means that the first nodal set shall intersect both components. Since the first nodal set can have only one component, it must cut through the hot spot curve, giving a contradiction. 
\end{obsv}  

Oftentimes, we will be interested in the situation that the starting point $x$ is inside $\Omega$, and the Brownian particle till time $t$ does not impact on the boundary. This is the same as $1 - \psi_{X\setminus \Omega}(t, x)$, where $X = \RR^n$ or $M$, as the case may be. 
\begin{notation}
    We denote the above mentioned survival probability by $q^\Omega_t(x)$. When there is no scope of confusion, we will drop the superscript $\Omega$, and just write $q_t(x)$. It is clear that $q_t(x) = 1 - p_t(x)$. 
\end{notation}
\subsubsection{Lower bound from Neumann Poincar\'{e}}\label{subsec:Neumann_Poin}
From Poincar\'{e}'s original proof \cite{P1890}, we know that 
$$
\int_\Omega v^2 \leq C \int_\Omega |\nabla v|^2,
$$
where $v \in \widetilde{H^1(\Omega)} := \{ u \in H^1(\Omega) : \int_\Omega u = 0\}$ and 
$$
C = |\diam(\Omega)|^2\frac{2^{n - 1} - 1}{n - 1}.
$$
This implies that 
\beq\label{ineq:neumann_est_diam}
\mu_2 \geq \frac{c_n}{\diam(\Omega)^2},
\eeq
where 
$$
c_n = \frac{n - 1}{2^{n - 1} - 1}.
$$

\subsubsection{Upper bound on $\mu_2$ due to Szeg\"{o}-Weinberger} The bound is given by 
\beq\label{ineq:Szeg_Wein}
\mu_2 \leq 4\pi^2\left( \frac{1}{\omega_n |\Omega|} \right)^{2/n}.
\eeq

\subsection{A recent heat kernel bound} 

Here we record some specific heat kernel bounds which would be suitable for later use. Let $p_\Omega(t, x, y)$ denote the Dirichlet heat kernel for the domain $\Omega$. Proposition 3.6 of \cite{BMW20} gives the following:
\begin{proposition} For all $\varepsilon \in (0, 1]$, we have that
\beq\label{ineq:heat_BMW}
    \int_\Omega p_\Omega(t, x, y) \; dy \leq e^{n/4}\frac{\sqrt{2}}{(8n)^{n/4}} \sqrt{ \frac{\Gamma(n)}{\Gamma(n/2)}  } \left( 1 +\frac{1}{\sqrt{\varepsilon}}\right)^{n/2}\; e^{- (1- \varepsilon)\lambda_1 t},
\eeq
for $t \geq 0$. 
\end{proposition}

\subsection{Some technical lemmata}\label{sec:tech_lem} Here we explicitly write out some ideas that are implicit in \cite{BW99, BB99}, which are useful for us, and we believe would be useful for future work in this direction. Consider the initial valued problem
\begin{align*}
(\pa_t - \Delta)u(t, x) & = 0, \text{    } x \in \Omega\\
\frac{\pa u(t, x)}{\pa \eta} & = 0, \text{    } x \in \pa \Omega\\
u(0, x) & = u_0(x), \text{    } x \in \Omega.
\end{align*}
We focus in particular on the situation where the initial condition $u_0(x) = \mathbbm{1}_S$ for some $S \subseteq \Omega$. Then we have that
\begin{lemma}\label{lem:u_S_Omega}
    Let $\Omega$ be a planar domain. For large $t$, the solution of the above diffusion process under $u_0(x) = \mathbbm{1}_S$ satisfies that \beq \label{ineq:Banuelos_Burdzy_est}
    \left|u(t, x) - |S|/|\Omega|\right| \leq Ce^{-\mu_2 t}
    \eeq
    for large $t$, where $C := C(|S|, |\Omega|)$. 
\end{lemma}
\begin{proof}
    We follow the idea in Proposition 2.1 of \cite{BB99}. Start by writing
    \beq\label{ineq:S_exp}
    u(t, x) = |S|/|\Omega| + \sum_{j \geq 1}e^{-t\mu_j}\left(\int_{S}u_j(x)\;dx\right)u_j(x),
    \eeq
and observe that $e^{\Delta} : L^2 \to L^\infty$ is bounded. If the norm of the above operator is $\leq C$, then
$$
e^{-\mu_j} u_j = e^{\Delta}u_j \leq C,
$$
implying that 
$$
|u_j(x)| \leq e^{\mu_j}.
$$
Applying this to (\ref{ineq:S_exp}) and using Weyl's law, we get the result. 
\end{proof}

Now, consider a domain $\Omega$ with a ``bottleneck'' part $B$ such that $\Omega \setminus B$ has two components, which we call $\Omega_L$ and $\Omega_R$ (for ``left'' part and ``right'' part respectively). More formally, $\Omega$ can be written as the union 
$$
\overline{\Omega} = \overline{\Omega_R} \cup \overline{\Omega_L} \cup \overline{B},
$$
where the interiors of the domains $\Omega_R, \Omega_L$ and $B$ do not intersect, and $B$ is the finite disjoint union of a collection of domains each of which is contained in the $\varepsilon$-tubular neighbourhood of a smooth curve of curvature $\leq \kappa$, where $\kappa$ is a universal constant. Let $\Gamma$ be a compact hypersurface contained $B$ which cuts $\pa\Omega \cap \pa B$ transversally (exactly two points in dimension $n = 2$). Let $\omega(t)$ be a reflected Brownian particle starting at $x$. Then we have the following:
 
\begin{lemma}\label{lem:u_equidist}
    Let $x \in \Omega_L$ (without loss of generality). Then, 
    \beq\label{ineq:Burdzy_Wer_est}
    \left| u(t, x) - |\Omega_R|/|\Omega| \right| \geq \left(|\Omega_R|/|\Omega| \right)^{ct}
    \eeq
    for large enough time $t$.
\end{lemma}
All the ideas in the proof are implicit in \cite{BW99}.
\begin{proof}
    Consider the diffusion process above with $u_0(x) = \mathbbm{1}_{\Omega_R}$. Then, 
    \begin{align*}
        u(t, x) & = \Prob\left( \omega(t) \in \Omega_R \right) = \Prob\left( \tau_\Gamma \leq t, \omega (t) \in \Omega_R\right) \\
        & = \frac{|\Omega_R|}{|\Omega|} \Prob 
        (\tau_\Gamma \leq t), 
    \end{align*}
    where $\tau_\Gamma$ is the first hitting time of $\Gamma$ by $\omega(t)$ with the last equality holding to arbitrary accuracy for large enough $t$. This follows by using the Markov property and Lemma \ref{lem:u_S_Omega} above by seeing that 
    $$
    \Prob(\tau_\Gamma \leq t, \omega(t) \in \Omega_R) = \int_{\Omega_R}p_N(t, x, y)\; dy = e^{t\Delta}\mathbbm{1}_{\Omega_R},
    $$ 
    and the latter converges to $\frac{|\Omega_R|}{|\Omega|}$ as $t \to \infty$.

The RHS of the above is clearly equal to
$
 \frac{|\Omega_R|}{|\Omega|} - \left(\frac{|\Omega_R|}{|\Omega|}\right)^{ct}  
$
for large enough $t$, where $c$ is defined by 
$$
c := \frac{\log(\Prob(\tau_\Gamma \geq t))}{t\log\left( |\Omega_R|/|\Omega|\right)}.
$$
Now, narrower the bottleneck of $\Omega$, smaller is the area of $\Gamma$, which implies that for large $t$, $\log\Prob(\tau_\Gamma \geq t)$ is bounded from below implying that $c$ is smaller. Also, is is clear that $\Area(\Gamma) \to 0 \implies c \to 0$. This proves the claim.

\end{proof}

\begin{lemma}\label{lem:eigen_exp}
Let $\lambda_1$ denote the ground state eigenvalue of a domain $\Omega$ where $\pa\Omega = \Gamma_1 \bigsqcup \Gamma_2$ with Dirichlet boundary condition on $\Gamma_1$ and Neumann boundary condition on $\Gamma_2$. Then we have, 
\beq\label{eq:eigenvalue_prob_char}
    \lambda_1 = -\lim_{t \to \infty} \frac{1}{t}\; \ln \left( \int_\Omega \Exp_x \left(\mathbbm{1}_{\omega^l(t)}\right)\right),
\eeq
where $\omega^l(t)$ is a Brownian path which is killed upon impact on $\Gamma_1$ and reflected upon impact on $\Gamma_2$, 
$$
	\mathbbm{1}_{\omega^l(t)} = 
	\begin{cases}
	1 & \text{if } \omega^l ([0, t]) \text{ does not strike } \Gamma_1,\\
	0 & \text{otherwise, }
	\end{cases}
$$
and the expectation is with respect to the corresponding Wiener measure.
\end{lemma}
\begin{proof}
We start with the usual Feynman-Kac formula for the mixed boundary problem. 
We let $\Delta_l$ denote the Laplacian with the Dirichlet boundary condition on on $\Gamma_1$ and Neumann boundary condition on $\Gamma_2$. 
We have that
\beq 
e^{t\Delta_l}f(x) = \Exp_x\left( f\left(\omega^l(t)\right)\right),
\eeq
which gives us that
\begin{align}\label{eq:exp_1_1}
    \left( e^{t\Delta_l}1, 1\right) & = \int_\Omega \Exp_x \left(\mathbbm{1}_{\omega^l(t)}\right).
\end{align}
By expanding in the eigenbasis $\varphi_j, j \in \NN$, we see that 
\begin{align*}
    e^{t\lambda_1}e^{t\Delta_l}1 & = \varphi_1 \int_\Omega \varphi_1 + e^{t(\lambda_1 - \lambda_2)}\varphi_2 \int_\Omega \varphi_2 + \dots,
\end{align*}
which gives
\begin{align*}
    e^{t\lambda_1}\left( e^{t\Delta_l}1, 1\right) = \left(\int_\Omega \varphi_1\right)^2 + R(t),
\end{align*}
where $R(t) \searrow 0$ as $t \nearrow \infty$. 
Using (\ref{eq:exp_1_1}) above and letting $t \to \infty$, we recover (\ref{eq:eigenvalue_prob_char}).
\end{proof}

\subsection{Estimates on the size of the second eigenvalue}\label{subsec:sec_eig_est}

When we put together (\ref{ineq:Banuelos_Burdzy_est}) and (\ref{ineq:Burdzy_Wer_est}) in the special case $S = \Omega_R$, we see that for a domain with a narrow enough bottleneck, we have that the first non-trivial Neumann eigenvalue is arbitrarily small, namely:
$$
\mu_2 \leq c.
$$

Conversely, using Lemma \ref{lem:eigen_exp}, we have the following:
\begin{proposition}\label{prop:nodal_set_hit}
    Let $\Gamma$ be 
    the nodal set for the eigenfunction $u_1$. 
    Suppose for all $x \in \Omega$, we have that 
    $$
    \Prob\left( \tau_\Gamma \leq t_0\right) \geq p_0,
    $$
    where $t_0, p_0$ are two positive constants. Then, $\mu_2$ cannot be arbitrarily small.
\end{proposition}
\begin{proof}
$\Gamma$ cuts out two nodal domains with mixed Dirichlet/Neumann boundary conditions. Then, 
$$
\Prob(\tau_\Gamma \leq t_0) \geq p_0 \implies \Prob(\tau_\Gamma \geq nt_0) \leq (1 - p_0)^n
$$
by application of the Markov property. Then an application of (\ref{eq:eigenvalue_prob_char}) gives us the claim. 
\end{proof}

Recall that in Theorem \ref{thm:Neumann_decay} we show via an argument involving the optional stopping theorem that absence of nodal set allows Neumann eigenfunctions to grow fast. Here, we show that in a domain with a narrow bottleneck, the first nodal line has to be near the bottleneck. This is in sharp contrast to the Dirichlet case as addressed in \cite{MS2021_2}. 

Suppose the nodal line is away from the bottleneck, whence its diameter needs to be above some constant, let's say $\varepsilon_0$. Then, we see that for all $x \in \Omega \setminus B$, $\Prob(\tau_\Gamma \leq t_0) \geq p_1$. Also, since any Brownian particle starting inside the bottleneck $B$ has high probability of escaping it, we see that $\Prob(\tau_{\Omega \setminus B} \leq t_0) \geq p_2$. By applying the Markovian property, for all $x\in \Omega$, we see that $\Prob(\tau_\Gamma \leq t_0) \geq p_0$. Now we bring in Proposition \ref{prop:nodal_set_hit} to finish the argument.

\subsection{Non-decay in bottle-neck domains: proof of Theorem \ref{thm:Neumann_decay}}   
Now, we take up a derivative question which on first glance might look somewhat unrelated: namely, that of tunneling of {\em low energy} Laplace eigenfunctions through narrow regions
in a domain. 
There is a substantial amount of recent literature on different 
variants and interpretations of the above question (see Section $7$ of \cite{GN2013} and references therein). We mention in particular the 
articles \cite{NGD2014}, 
\cite{vdBBo95} and \cite{BD92}. Inherent in all the above mentioned references seems to be the following heuristic (physical) fact: there is a sharp decay of a Dirichlet eigenfunction in a ``narrow tunnel-shaped'' (or horn-shaped) region when the wavelength of the eigenfunction is larger than the width of the tunnel. In other words, Dirichlet eigenfunctions cannot tunnel effectively through subwavelength openings. This leads to the localization of low frequency eigenfunctions in relatively thicker regions of a domain. Intuitively, this can be ascribed to the loss of energy through the walls surrounding the narrow tunnel. A version of the above has been proved in Theorem 1.6 of \cite{GM2021}. Now we consider again the case of an ``octopus'' domain, with a central large blob and narrow tentacles coming out of it, but this time with Neumann boundary conditions. We will assume that the domain has regular enough boundary for RBM to make sense, and prove that Neumann eigenfunctions do not need to decay inside the tentacles at all unless the nodal set enters the tentacle. Observe that intuitively, one does not really expect a decay as there is no loss of energy, as there is no heat conduction across the layers of the tunnel. A similar heuristic is at the heart of why one cannot expect domain monotonicity of Neumann eigenvalues.

\vspace{5mm}
 \begin{figure}[ht]
\centering
\includegraphics[scale=0.2]{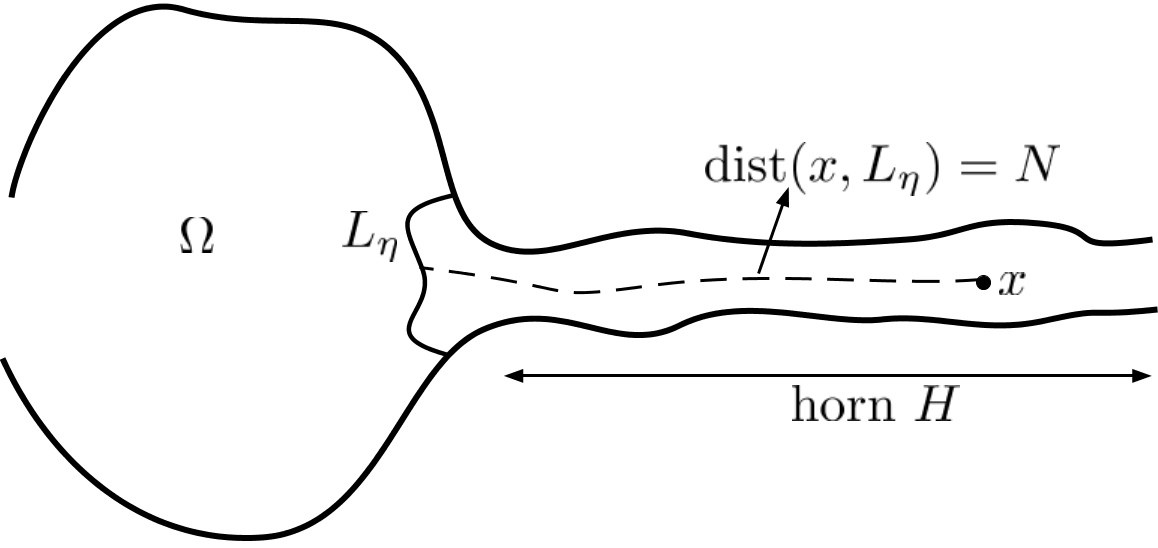} 
\caption{Domain $\Omega$ with a horn/tentacle $H$}
\label{fig:domain_horn}
\end{figure}
     
    \vspace{5mm}
Now we start the proof of Theorem \ref{thm:Neumann_decay}. 
The main idea is to use the RBM, and a version of Doob's optional stopping theorem, as in Theorem \ref{thm:opt_stop_level} above.
\begin{proof}
Recall that the solution to the backward heat equation $u(t, y) := e^{\mu t}\varphi(y)$ gives rise to a non-negative martingale $u(t, \omega(t))$ in the nodal domain containing $x$. The time $\tau$ when a Brownian particle emanating from $x$ hits $L_\eta$ (see Figure \ref{fig:domain_horn}) is a stopping time, and so is $t^* := \min\{t, \tau\}$, where $t > 0$ is some real number. Then by Doob's optional stopping theorem, we have that 
\begin{align*}
\varphi(x) & = \Exp_x\left( e^{\mu.0}\varphi(x)\right) = \Exp\left( e^{\mu t^*}\varphi(\omega(t^*))\right)\\
& = \Exp\left( e^{\mu \tau}\varphi(\omega(\tau))\chi_{\tau \leq t}\right) + \Exp\left( e^{\mu t}\varphi(\omega(t))\chi_{\tau \geq t}\right)\\
& \geq \Exp\left( e^{\mu\tau}\varphi(\omega(\tau))\chi_{\tau \leq t}\right).
\end{align*} 
Letting $t \nearrow \infty$, it is clear from Markov's inequality that
\begin{align*}
    \varphi(x) & \geq \eta\; \Exp\left( e^{\mu \tau}\right) \geq \eta\;e^a\Prob\left( e^{\mu \tau} \geq e^a\right) \text{ for all real numbers } a \\
    & = \eta\;e^a\Prob\left( \tau \geq a/\mu\right) = \eta\;
    e^{\mu N^2}
    \Prob \left( \tau \geq 
    N^2\right).
\end{align*}
Since a typical Brownian particle covers a distance $\sim \sqrt{t}$ within time $t$, by an argument similar to Theorem \ref{thm:Dirichlet_level_set} our claim now follows.
\end{proof}

\begin{remark}
    Observe that the above proof is quite robust in the sense that we have not used the fact that the horns are ``narrow'' in any essential way. This is again, in contrast with Theorem 1.6 of \cite{GM2021}, where narrowness of the tunnel is of critical importance to force decay. So it is quite conceivable that the argument in the above proof could be useful in more diverse situations. 
\end{remark}

\subsection{Hot spot constant: proof of Theorem \ref{thm:hot_spot_constant_high_mu}} Start a RBM from a point $x \in \Omega$, and look at the particle at time $t$. There are two options: the particle strikes $\pa\Omega$ within time $t$, or not. Let $\tau_\Omega$ denote the first ``reflecting time'', that is, the first time that the Brownian particle strikes the boundary $\pa\Omega$. Then one has
\begin{align*}
    e^{-\mu_2 t}u(x) & = \Exp_x \left( u(\omega(t) \mathbbm{1}_{\tau_\Omega > t}\right) + \Exp_x \left( u(\omega(t) \mathbbm{1}_{\tau_\Omega \leq t}\right)\\
    & \leq \| u\|_{L^\infty(\Omega)} q_t(x) + 
    \Exp_x \left( u(\omega(t) \mathbbm{1}_{\tau_\Omega \leq t}\right). 
\end{align*}
 
By the strong Markov property, we have that 
\begin{align*}
    \Exp_x \left( u(\omega(t) \mathbbm{1}_{\tau_\Omega \leq t})\right) & = \Exp_x \left( \Exp_{\omega(\tau_\Omega)} u(\omega(t - \tau_\Omega)) \mathbbm{1}_{\tau_\Omega \leq t}\right)\\
     & = \Exp_x\left( e^{-\mu_2(t - \tau_\Omega)} u(\omega(\tau_\Omega))\mathbbm{1}_{\tau_\Omega \leq t} \right) \\
     & \leq \| u\|_{L^\infty(\pa\Omega)}e^{-\mu_2t}\Exp_x \left( e^{\mu_2\tau_\Omega} \mathbbm{1}_{\tau_\Omega \leq t}\right).
\end{align*}

So, if $x$ is an interior hot spot for $u$, that is, $u(x) = \| u \|_{L^\infty(\Omega)}$, we have that 
$$
e^{-\mu_2t}\| u \|_{L^\infty(\Omega)} \leq \| u \|_{L^\infty(\Omega)} q_t(x) + \| u \|_{L^\infty(\pa\Omega)}e^{-\mu_2t} \Exp_x\left( e^{\mu_2\tau_\Omega}\mathbbm{1}_{\tau_\Omega \leq t} \right),
$$
which implies that 
$$
e^{-\mu_2t} \leq q_t(x) + \frac{\|u \|_{L^\infty(\pa\Omega)}}{\|u \|_{L^\infty(\Omega)}}e^{-\mu_2t}\Exp_x \left( e^{\mu_2\tau_\Omega}\mathbbm{1}_{\tau_\Omega \leq t} \right),
$$
or
\beq
\frac{\|u \|_{L^\infty(\pa\Omega)}}{\|u \|_{L^\infty(\Omega)}} \geq \frac{ e^{-\mu_2t} - q_t(x) }{ e^{-\mu_2t} \Exp_x \left( e^{\mu_2\tau_\Omega} \mathbbm{1}_{\tau_\Omega \leq t}\right) } = \frac{ 1 - e^{\mu_2t} q_t(x) }{ \Exp_x\left( e^{\mu_2\tau_\Omega}\mathbbm{1}_{\tau_\Omega \leq t}\right)}.
\eeq


Calculating further, we see that, 
\begin{align*}
    \frac{\|u \|_{L^\infty(\pa\Omega)}}{\|u \|_{L^\infty(\Omega)}} & \geq \frac{ e^{-\mu_2t} - q_t(x) }{ \Exp_x \left( e^{\mu(\tau_\Omega - t)} \mathbbm{1}_{\tau_\Omega \leq t}\right) } \\
    & \geq \frac{
    e^{-\mu_2t} - c_2\zeta_2(\varepsilon)e^{-(1 - \varepsilon)\lambda_1t}}{ \Exp_x \left( e^{\mu_2(\tau_\Omega - t)} \mathbbm{1}_{\tau_\Omega \leq t}\right)} \quad \text{  (using 
    (\ref{ineq:heat_BMW}))}
\end{align*}
where $\zeta_n(s) := e^{n/4}\frac{\sqrt{2}}{(8n)^{n/4}} \sqrt{ \frac{\Gamma(n)}{\Gamma(n/2)}  } \left( 1 +\frac{1}{\sqrt{s}}\right)^{n/2}$.


Following a calculation in \cite{MPW}, we see that for $t \geq 0$ and $ x \in \overline{\Omega}$, we have almost surely
\begin{align*}
    \int_0^t \mu_2e^{\mu_2s} \mathbbm{1}_{\tau_\Omega > s} \; ds & = \int_0^{\tau_\Omega \wedge t} \mu_2e^{\mu_2s} \; ds = e^{\mu_2(\tau_\Omega \wedge t)} - 1\\
    & = e^{\mu_2\tau_\Omega} \mathbbm{1}_{\tau_\Omega \leq t} + e^{\mu_2t}\mathbbm{1}_{\tau_\Omega > t} - 1.
\end{align*}

Taking expectations and rearranging one gets the following:
\beq\label{eq:MPW_21}
\Exp\left( e^{\mu_2\tau_\Omega} \mathbbm{1}_{\tau_\Omega \leq t} \right)  = 1 - e^{\mu_2t}q_t(x) + \int_0^t \mu_2e^{\mu_2s}q_s(x) \; ds.
\eeq

On calculation, one sees that 
\begin{align*}
    e^{-\mu_2t}\Exp_x\left( e^{\mu_2\tau_\Omega} \mathbbm{1}_{\tau_\Omega \leq t}\right) & = 
    e^{-\mu_2t} - q_t(x) + 
     \zeta_2(\varepsilon) \mu_2e^{-\mu_2t} \frac{ e^{s(\mu_2- (1 - \varepsilon)\lambda_1)}}{\mu_2- (1 - \varepsilon)\lambda_1}\bigg|_0^t\\
    & = 
    e^{-\mu_2t} - q_t(x) + \frac{  \zeta_2(\varepsilon) \mu_2e^{-\mu_2t} }{ \mu_2- (1 - \varepsilon)\lambda_1 } \left( e^{t(\mu_2- (1 - \varepsilon)\lambda_1)} - 1\right).
\end{align*}

Putting things together, we see that 
\begin{align*}
    \frac{\| u \|_{L^\infty (\pa\Omega)}}{\| u \|_{L^\infty(\Omega)}} & \geq \frac{ e^{-
    \mu_2t} -  \zeta_2(\varepsilon) e^{ -(1 - \varepsilon) \lambda_1 t} }{ 
    e^{-\mu_2t} - q_t(x) + \frac{ \zeta_2(\varepsilon) \mu_2e^{-\mu_2t}}{\mu_2- (1 - \varepsilon)\lambda_1} \left( e^{t(\mu_2- (1 - \varepsilon)\lambda_1)} - 1\right)}.
\end{align*}
In the regime $ t = c\mu_2^{-1}$,  letting $\sigma := \lambda_1/\mu_2$, and writing $a:= (1 - (1 - \varepsilon)\sigma)$ we see that 
\beq\label{ineq:hot_spot_opti}
     \frac{\| u \|_{L^\infty ( \Omega)}}{\| u \|_{L^\infty(\pa \Omega)}}    \leq  1 + \frac{ \zeta_2(\varepsilon) \left(e^{ca}  - 1\right)}{ a\left( 1 - \zeta_2(\varepsilon)e^{ca}\right)},
\eeq
where $\varepsilon\in (0, 1-\frac{1}{\sigma})$. The expression on the right decreases as $c \nearrow \infty$. Then (\ref{ineq:hot_spot_opti}) reduces to 
\beq\label{ineq:hot_spot_c_large}
\frac{\| u \|_{L^\infty ( \Omega)}}{\| u \|_{L^\infty(\pa \Omega)}}    \leq  1 + \frac{\zeta_2(\varepsilon)}{\sigma(1- \varepsilon) - 1}.
\eeq
In particular, for domains with high  values of $\lambda/\mu$, the hot spots constant is arbitrarily close to $1$. 

\subsection{Acknowledgements} The first author would like to thank Harsha Hutridurga for teaching him the material in Subsubsection \ref{subsec:Neumann_Poin}. 
The first author's research was partially supported by SEED Grant RD/0519-IRCCSH0-024. The second named author would like to thank the Council of Scientific and Industrial Research, India for funding which supported his research. The authors would like to deeply thank Stefan Steinerberger, Sugata Mondal, Siva Athreya and  Vivek Borkar for illuminating conversations. Finally, the authors wish to thank Indian Institute of Technology Bombay for providing ideal working conditions.

\bibliographystyle{alpha}
\bibliography{references}

\end{document}